\documentclass[12pt]{amsart}
\usepackage[includeheadfoot, left=1in, right=1in, top=.8in, bottom=.8in]{geometry}

\usepackage{times}
\usepackage[english]{babel}

\usepackage{mathrsfs}
\usepackage{amsmath}
\usepackage{amsfonts}
\usepackage{amssymb}
\usepackage{amscd}
\usepackage{amsxtra}
\usepackage{amsthm}
\usepackage{color}
\usepackage[dvipsnames]{xcolor}
\usepackage[linktocpage=true]{hyperref}
\definecolor{linkcolor}{rgb}{0,0,0.7}
\definecolor{urlcolor}{rgb}{0,0,0.4}
\definecolor{citecolor}{rgb}{0.4,0.1,0}

\newcommand{\Const}{\text{Const}}

\newcommand{\diam}{\text{diam}}
\newcommand{\diff}{\text{d}}

\newcommand{\inj}{\text{inj}}

\newcommand{\Ric}{\text{Ric}}

\newcommand{\supp}{\text{supp}}

\newcommand{\dvol}{\ \text{dvol}}
\newcommand{\Vol}{\text{Vol}}

\newcommand{\wto}{\rightharpoondown}

\newcommand{\R}{\ensuremath{\mathbb{R}}}

\newcommand{\Cc}{\ensuremath{\mathcal{C}}}

\newcommand{\Rr}{\ensuremath{\mathcal{R}}}
\newcommand{\Ss}{\ensuremath{\mathcal{S}}}

\newcommand{\gm}{\ensuremath{\gamma}}
\newcommand{\af}{\ensuremath{\alpha}}

\newcommand{\omg}{\ensuremath{\omega}}
\newcommand{\Omg}{\ensuremath{\Omega}}

\newcommand{\lbd}{\ensuremath{\lambda}}
\newcommand{\Lbd}{\ensuremath{\Lambda}}
\newcommand{\fai}{\ensuremath{\varphi}}
\newcommand{\Fai}{\ensuremath{\varPhi}}
\newcommand{\epsl}{\ensuremath{\varepsilon}}

\usepackage{accents}
\newcommand{\ubar}[1]{\underaccent{\bar}{#1}}

\newcommand{\ptl}{\ensuremath{\partial}}

\newtheorem{thm}{Theorem}

\newtheorem{lem}[thm]{Lemma}

\theoremstyle{remark}

\newtheorem{rmk*}{Remarks}

\theoremstyle{definition}

\address{Department of Mathematical Sciences\\
	Tsinghua University\\
	Beijing 100084, P.R. China}
\email[]{dch17@mails.tsinghua.edu.cn}

\begin{document}
	
\title[Convergence of manifolds]{Convergence of manifolds under some\\ $L^p$-integral curvature conditions}
\author{Conghan Dong}
\date{}
\maketitle
\begin{abstract}
	Let $\Cc(\Rr,n,p,\Lbd,D,V_0)$ be the class of compact $n$-dimensional Riemannian manifolds with finite diameter $\leq D$, non-collapsing volume $\geq V_0$ and  $L^p$-bounded $\Rr$-curvature condition $\|\Rr\|_{L^p}\leq \Lbd$ for some $p>\frac n2$. Let $(M,g_0)$ be a compact Riemannian manifold and $\Cc(M,g_0)$ the class of manifolds $(M,g)$ conformal to $(M,g_0)$. In this paper we use $\epsl$-regularity to show a rigidity result in the conformal class $\Cc(S^n,g_0)$ of standard sphere under $L^p$-scalar rigidity condition. Then we use harmonic coordinate to show $C^{\af}$-compactness of the class $\Cc(K,n,p,\Lbd,D,V_0)$ with additional positive Yamabe constant condition, where $K$ is the sectional curvature, and this result will imply a generalization of Mumford's lemma. Combining these methods together we give a geometric proof of $C^{\af}$-compactness of the class $\Cc(K,n,p,\Lbd,D,V_0)\cap \Cc(M,g_0)$. By using Weyl tensor and a blow down argument, we can replace the sectional curvature condition by Ricci curvature and get our main result that the class $\Cc(\Ric,n,p,\Lbd,D,V_0)\cap \Cc(M,g_0)$ has $C^{\af}$-compactness.
\end{abstract}
\tableofcontents

\section{Introduction}
The Gromov-Hausdorff convergence theorems of Riemannian manifolds have been studied for a long time since M. Gromov stated a striking result about the $C^{1,\af}$-compactness of the class $\Cc(K,n,\infty,\Lbd,D,V_0)$, which consists of all compact $n$-dimensional Riemannian manifolds with diameter $\leq D$, volume $\geq V_0$ and  $L^p$-bounded sectional curvature condition $\|K\|_{L^p}\leq \Lbd$ for some $p>\frac n2$, see \cite{Gromov:1981} and also \cite{Cheeger:1970}. There are lots of articles in this subject later and one can refer to a survey \cite{Petersen:1997}.

In \cite{Anderson:1990}, M. Anderson used harmonic coordinate to prove a similar compactness but with sectional curvature replaced by Ricci curvature and volume lower bound replaced by injectivity radius, and he also showed $C^{1,\af}$-compactness of the class $\Cc(\Ric,n,\infty,\Lbd,D,V_0)$ with additional $L^{\frac n2}$-smallness sectional curvature condition $\|K\|_{L^{\frac n2}}\leq \epsl$ for small enough $\epsl>0$. By generalizing the volume comparison theorem from the condition $\Ric\geq (n-1)\lbd$ to $\|(\Ric-(n-1)\lbd g)_{-}\|_{L^p}\leq \epsl$ for some $p>\frac n2$ and $\lbd\leq 0$ in \cite{PetersenWei:1997}, P. Petersen and G. Wei got a $C^{\af}$-compactness with the condition $|\Ric|\leq \Lbd$ generalized to the condition $\|\Ric\|_{L^p}\leq \Lbd$ and smallness of $\|(\Ric-(n-1)\lbd g)_{-}\|_{L^p}$. In general without additional geometric conditions, it is easy to see the class $\Cc(\Ric,n,p,\Lbd,D,V_0)$ does not admit $C^{\af}$-compactness.

When given a compact $n$-dimensional Riemannian manifold $(M,g_0)$ and consider the case in a fixed conformal class $\Cc(M,g_0)$, which consists of all metrics $(M,g)$ conformal to $(M,g_0)$, we can have better results. In \cite{Gursky:1993}, M. Gursky proved the following theorem that in the $L^p$-bounded sectional curvature case $\Cc(K,n,p,\Lbd,D,V_0)\cap \Cc(M,g_0)$ has $C^{\af}$-compactness.
\begin{thm}\label{L^p-section}
	Assume that $(M,g_0)$ is a compact Riemannian manifold with dimension $n\geq 3$, then given constants $p>\frac n2$, $\Lbd,V_0,D>0$, the space 
	$$ 
	\{(M,g)\in \Cc(M,g_0): \int_M|K_g|^p\dvol_g\leq \Lbd, Vol(M,g)\geq V_0, \diam_gM\leq D \}
	$$ 
	is compact in the $C^{\alpha}$-topology for any $0<\af<2-\frac n{p}$.
\end{thm}
See also the related results in \cite{ChangYang:1990}.
In \cite{LiZhou:2017}, Y. Li and Z. Zhou considered the case with $L^p$-bounded scalar curvature and studied the bubble tree convergence of class $\Cc(R,n,p,\Lbd,D,V_0)\cap \Cc(M,g_0)$, where $R$ is the scalar curvature, and they proved Gursky's result as a corollary. Also additional calculation based on the bubble tree convergence can in fact imply that the class $\Cc(M,g_0)\cap\Cc(\Ric,n,p,\Lbd,D,V_0)$ has $C^{\af}$-compactness. The main method used in those papers is by analyzing the scalar curvature equation deeply. 

In this paper we consider these problems through a more geometric approach used in Anderson's compactness results. In other words we will give a proof of Gursky's result by combining $\epsl$-regularity of the scalar curvature equation together with harmonic coordinate techniques.
This method can be easily generalized to the case with $L^p$-bounded Ricci curvature condition. Since in dimension $3$ there is no difference between the sectional curvature and Ricci curvature condition, our main result is the following.
\begin{thm}\label{L^p-bounded ricci}
	Assume that $(M,g_0)$ is a compact Riemannian manifold with dimension $n\geq 4$, then given constants $p>\frac n2$, $\Lbd,V_0,D>0$, the space 
	$$\{(M,g)\in \Cc(M,g_0): \int_M|\Ric_g|^p\dvol_g\leq \Lbd, Vol(M,g)\geq V_0,\diam_gM\leq D \}$$
	is compact in the $C^{\alpha}$-topology for any $0<\af<2-\frac np$.
\end{thm}
The proof is a usual contradiction argument by blowing up the harmonic radius. Note that unlike before, in the Ricci curvature case we can only get a Ricci-flat limit space. But we have the conformal invariance of $L^{\frac n2}$-norm of Weyl tensor, and Ricci-flat together with Weyl-flat implies flat. So when the limit space comes from blowup at some point, we can use Weyl tensor to show that the Ricci curvature case dose not make big difference from the sectional curvature case when restricted in a conformal class. For the remaining case, we will use a blow down method to show the infinity of the limit space lies in a sufficiently small neighborhood of some point in a compact manifold, which together with a gap lemma can imply such limit space is $\R^n$, which gives the desired contradiction.

Before proving the main results, we will first examine these techniques separately. Our first result is about rigidity in a conformal class of $S^n$ under scalar curvature rigidity condition.
\begin{thm}\label{L^p-constant scalar}
	Assume $(S^n,g_0)$ is the standard sphere with $n\geq 3$. Given $\epsl>0$, there exists a $\delta(n,\epsl)>0$ such that if $ (S^n,g)\in \Cc(S^n,g_0) $ with same volume $$\Vol_gS^n=\Vol_{g_0}S^n$$ and $L^p$-constant scalar curvature condition
	$$ \int_{S^n}|R_{g}-n(n-1)|^p \dvol_{g}\leq \delta$$ for some $ p>\frac{n}{2} $,
	then the $C^{\af}$-Gromov-Hausdorff distance between $(S^n,g)$ and $(S^n,g_0)$ is at most $\epsl$ for any $0<\af<2-\frac np$.
\end{thm}
If we write $g=u^{\frac{4}{n-2}}g_0$, then $u$ satisfies an elliptic PDE from the scalar curvature equation, see in Section \ref{L^p constant on S^n} below, and we have $\epsl$-regularity lemma which states that we can have uniform bounded $W^{2,p}$-norm of $u$ whenever we have bounded $L^p$-norm and small $L^{\frac{2n}{n-2}}$-norm of $u$ on covering domains with uniform size. This rigidity result can then be derived if there are no such $L^{\frac{2n}{n-2}}$-concentration points. Otherwise we will blow up at such point by the pullback of scalar dilation over $\R^n$ through the stereographic projection. After applying such induced blowup maps which are diffeomorphisms on $S^n$, we can get desired convergence.

Our second result is by using harmonic coordinate to get a compactness result under $L^p$-bounded sectional curvature and additional positive Yamabe constant condition. For the definition and properties of Yamabe constant, you can see below in Section \ref{positive-yamabe-constant and sobolev constant}.
\begin{thm}\label{positive-yamabe}
	Given constants $n\geq 3$, $p>\frac n2$, $\Lbd, D, V_0>0$, and a sequence of compact $n$-dimensional Riemannian manifolds $$(M_i,g_i)\in \Cc(K,n,p,\Lbd,D,V_0).$$ If there exists a uniform positive radius $s_0$ and a constant $Y_0>0$ such that for all $x\in M_i$ the Yamabe constant $$Y(B_{s_0}^{g_i}(x),g_i)\geq Y_0,$$
	then there is a subsequence converges in the $ C^{\af}$-topology for any $0<\af<2-\frac{n}{p}$.
\end{thm}
The proof of such compactness result is standard. By blowup of the harmonic radius, we get a flat complete non-compact limit manifold, and as usual done in \cite{Anderson:1990} we need to show the limit space is isometric to Euclidean space to get a contradiction. So it is suffice to show the limit flat manifold has maximal volume growth, and in our cases we usually use uniform Sobolev inequality, which is ensured by assumed uniformly positive Yamabe constant, to get such volume growth.

As a corollary, we can prove the generalization of Mumford's lemma for locally conformal flat manifolds.
\begin{thm}\label{Mumford's lemma}
	Given constants $n\geq 3$, $p>\frac n2$, $\Lbd, D, V_0>0$, and a sequence of compact $n$-dimensional Riemannian manifolds $$(M_i,g_i)\in \Cc(K,n,p,\Lbd,D,V_0).$$ If $(M_i,g_i)$ are non-simply connected locally conformal flat manifolds and the length of all nontrivial elements in $\pi_1(M_i)$ has a uniform positive lower bound, then there is a subsequence converges in the $ C^{\af}$-topology for any $0<\af<2-\frac{n}{p}$.
\end{thm}

This paper is organized as following. In next section we recall and prove some basic results of PDE and geometry which are needed later. The remaining sections are all devoted to prove results as entitled. In the following, we will use the same notation $C$, which may be different from line to line, to denote a uniform constant not depending on the sequence. And we will often write the metric $g$ on super- or sub-script of a geometric subject to emphasize the considered metric $g$.

\textbf{Acknowledgment}. The author would like to express his deep gratitude to his advisor Professor Yuxiang Li for helpful ideas and discussions.

\section{Preliminary results}

\subsection{Estimates of PDE}
We first recall the $L^p$-estimate in elliptic PDE, see e.g. \cite{GilbargTrudinger:2001} for details.
\begin{lem}[$L^p$-estimate]\label{L^p-estimate}
	Over smooth domain $\Omg\subset \R^n$ for equation of $u\in W^{1,2}(\Omg)$
	$$-a^{ij}\ptl_i\ptl_ju+b^i\ptl_iu=f,$$
	with $$C_1^{-1}\leq a^{ij}\leq C_1,\ \|a^{ij}\|_{C^0}+ \|b^i\|_{L^q}\leq C_2$$ for some $q>n$, then for $p<q$
	we have interior estimate over $\Omg'\subset\subset\Omg$
	$$\|u\|_{W^{2,p}(\Omg')}\leq C(\Omg',\Omg,C_1,C_2)(\|u\|_{L^p(\Omg)}+\|f\|_{L^p(\Omg)}).$$
	If also $u\in W^{1,2}_0(\Omg)$ then we have global estimate
	$$\|u\|_{W^{2,p}(\Omg)}\leq C(\Omg,C_1,C_2)(\|u\|_{L^p(\Omg)}+\|f\|_{L^p(\Omg)}).$$
\end{lem}
With the $L^p$-estimate and an iteration method, we can prove an $\epsl$-regularity lemma which will be used later for the scalar curvature equation. One  can refer to \cite{LiZhou:2017} for a detailed proof.
\begin{lem}[$\epsl$-regularity]\label{epsl-regularity}
	Over $ B_1\subset \R^n, $ consider the equation of $0\leq u\in W^{1,2}(B_1)$
	\[ -a^{ij}\ptl_i\ptl_ju+b^i\ptl_iu=f_1u+f_2u^{\frac4{n-2}}u, \]
	with $$C^{-1}_1\leq a^{ij}\leq C_1,\ \|a^{ij}\|_{C^0}+ \|b^i\|_{L^q}\leq C_2$$ for some $q>n$, and $$ \int_{B_1} |f_2|^pu^{\frac{2n}{n-2}}\diff x\leq C_3,$$  for some $\frac n2<p<q$. Then there exists a small positive constant $\epsl_0=\epsl_0(n,p,C_3)$ such that if $$\int_{B_1}|f_1|^n\diff x\leq C_4\ or\ \int_{B_1}|f_1|^p\diff x\leq \epsl_0$$ and \[ \int_{B_1}|u|^{\frac{2n}{n-2}}\diff x\leq \epsl_0 ,\] then \[ \|u\|_{W^{2,p}(B_{\frac12})}\leq C(C_1,C_2,C_3,C_4)\|u\|_{L^{\frac{2n}{n-2}}(B_1)}. \]
\end{lem}

\subsection{Harmonic radius and compactness}
Recall that we say a sequence of pointed complete Riemannian manifolds $$(M_i,g_i,p_i)\to (M,g,p)$$ in the $C^{m,\af}$-topology with $m\geq 0$ and $\af\in(0,1)$, if for every $r>0$ there exists a domain $B_r^g(p)\subset \Omg\subset M$ and embeddings $f_i:\Omg\to M_i$ for all large $i$ such that $B_r^{g_i}(p_i)\subset f_i(\Omg)$ and $f_i^{-1}(p_i)\to p$, $f_i^*g_i\to g$ on $\Omg$ in the $C^{m,\af}$-topology. Similarly we have the $W^{k,p}$-convergence. Note that the $C^{\af}$-convergence of metrics ensures the convergence of distance and thus convergence of pointed geodesic balls together with their volumes. 

For a Riemannian $n$-manifold $(M,g)$, and a constant $q>n$, we say that the $W^{1,q}$-harmonic norm on the scale of $r$ of $A\subset (M,g)$: $$\|A\subset(M,g)\|_{W_{har}^{1,q},r}\leq Q,$$
if we can find charts $$\fai_s:D_r(0)\subset \R^n\to U_s\subset M,$$
such that
\begin{itemize}
	\item[1)] Every ball $B_{\frac1{10}e^{-Q}r}^g(p)$, $p\in A$ is contained in some $U_s$;
	\item[2)] $|D\fai_s|\leq e^Q$ on $D_r(0)$ and $|D\fai_s^{-1}|\leq e^Q$ on $U_s$;
	\item[3)] $r^{1-\frac nq}\|Dg_{s\cdot\cdot}\|_{L^q}\leq Q$;
	\item[4)] $\fai_s^{-1}:U_s\to \R^n$ is harmonic with respect to the metric $(M,g)$.
\end{itemize}
The $W^{1,q}$-harmonic radius $\epsl(x)$ at $x$ is the maximal radius satisfying above conditions around $x$, and the $W^{1,q}$-harmonic radius of $(M,g)$ is the minimum of $\epsl(x)$ for all $x\in M$. Similarly we have the $C^{1,\af}$-harmonic coordinate and radius. Then we have the Gromov compactness theorem. For a proof and the related results see e.g. \cite{Anderson:1990, AndersonCheeger:1992, HebeyHerzlich:1997}.
\begin{lem}\label{compactness}
	For given constants $Q>0, n\geq 2, q>n$ and $r_0>0$, then the class of complete pointed $n$-dimensional Riemannian manifolds $(M,g,p)$ with $\|(M,g)\|_{W^{1,q}_{har},r_0}\leq Q$ is closed in the pointed $W^{1,q}$-topology and compact in the pointed $C^{\af}$-topology for all $\af<1-\frac nq$.
\end{lem}
For later use let us recall some basic lemmas in \cite{Anderson:1990}, where a blowup argument and splitting theorem are used to show that a lower bound of injectivity radius implies a lower bound of harmonic radius.
\begin{lem}\label{injectivity}
	If $(M_i,g_i,p_i)$ are $n$-dimensional Riemannian manifolds with $$|\Ric_{g_i}|\leq \Lbd,$$ and for any $r>0$ and all large $i$ the injectivity satisfying $\inj_{g_i}B_r(p_i)\geq i_0(r)>0$, then there is a subsequence converges in the pointed $C^{1,\af}$-topology for any $0<\af<1$.
\end{lem}
\begin{lem}\label{local-compactness}
	For a sequence of pointed complete $n$-dimensional Riemannian manifolds $(M_i,g_i,p_i)$ which converges in the Gromov Hausdorff topology to a pointed metric space $(M_\infty,d_\infty,p_\infty)$. Suppose that $$|\Ric_{g_i}|\leq \Lbd$$ and the injectivity radius of $(M_i,g_i)$ over $\bar B_{r_2}^{g_i}(p_i)\setminus B_{r_1}^{g_i}(p_i)$ for any positive $r_1<r_2$ is lower bounded by a uniform constant $i_0(r_1,r_2)>0$. Then $M_\infty\setminus\{p_\infty\}$ is a smooth manifold with $C^{1,\af}$-Riemannian metric $g_\infty$ for all $0<\af<1$,  which is compatible with the distance $d_\infty$, and also $$(M_i,g_i)\to (M_\infty,g_\infty)\ in\ C^{1,\af}_{loc}(M_\infty\setminus\{p_\infty\}).$$
\end{lem}
Also by the similar blowup argument, there is a gap lemma in \cite{Anderson:1990}.
\begin{lem}\label{gap-lem-volume}
	Let $(M,g)$ be a complete Ricci flat $n$-manifold. There is an $\epsl=\epsl(n)>0$ such that if $$\frac{\Vol_gB_r}{\omg_nr^n}\geq 1-\epsl,$$then $M$ is isometric to $\R^n$.
\end{lem}
\subsection{Sobolev inequalities and volume ratios}
Recall that for a domain $\Omg$ in a Riemnnian manifold $(M,g)$, we have the Dirichlet-Sobolev functional for $\fai\in W^{1,2}_{0}(\Omg,g)$ $$SD(\Omg,\fai):=\frac{\int_{B_R(p)}|\nabla_g\fai|^2\dvol_g}{\left(\int_{B_R(p)}\fai^{\frac{2n}{n-2}}\dvol_g\right)^{\frac{n-2}n}},$$
and the Dirichlet-Sobolev constant $$SD(\Omg,g):=\inf_{\substack{\fai\in W^{1,2}_0(\Omg)\\ \fai\neq 0}}SD(\Omg,\fai).$$
We will need a fact that Sobolev inequality gives a volume growth control.
\begin{lem}\label{volume-ratio}
	On a complete Riemannian manifold $(M,g)$, for a geodesic ball $B_R^g(p)$ with $\ptl B_R^g(p)\neq \emptyset$ if there exists a positive $\mu>0$ such that the Dirichlet Sobolev constant $SD(B_R^g(p),g)\geq\mu$,
	then there exists $C(n,\mu)>0$ such that for any $0<r\leq R$ $$\frac{\Vol_gB^g_r(p)}{\omg_nr^n}\geq C(n,\mu).$$
\end{lem}
\begin{proof}
	Consider the volume ratio function $$\Theta(r):= \begin{cases}
	\frac{\Vol_gB_r(p)}{\omg_nr^n},\ & 0<r\leq R\\
	1,\ & r=0
	\end{cases}.$$
	Then $\Theta(r)$ is continuous and positive on $[0,R]$. Say $\Theta(r)$ takes its minimum at $r=r_0$. If $r_0=0$ then we can take $C(n,\mu)=1$. So assume $r_0>0$ and then from $\Theta(\frac{r_0}2)\geq \Theta(r_0)$ we have $$\Vol_gB_{\frac{r_0}2}(p)\geq 2^{-n}\Vol_gB_{r_0}(p).$$
	Now take a test function $\fai(x)=\fai(d_g(p,x))$ with $$0\leq\fai\leq 1, \ \supp\fai\subset B_{r_0}(p),\ \fai|_{B_{\frac{r_0}2}}=1,\ |\nabla_g\fai|\leq \frac4{r_0},$$
	then $$\mu\cdot (\Vol_gB_{\frac{r_0}2}(p))^{\frac{n-2}n}\leq \Vol_gB_{r_0}(p)\cdot \frac{16}{r_0^2},$$
	which together with above inequality implies that $$\Theta(r_0)\geq \frac1\omg_n\cdot\left(\frac{\mu}{2^{n+2}}\right)^{\frac n2}=:C(n,\mu)>0.$$
\end{proof}
Now we recall some basic facts about Sobolev constant of $\R^n$, and one can see the related results in R. Schoen and S.T. Yau's book \cite{SchoenYau:1994}.
Over Euclidean space $\R^n$, we know the optimal Sobolev constant is $$SD(\R^n)=n(n-1)\omg_n^{\frac2n},$$
which can be achieved by the rotational symmetric functions $(a+br^2)^{\frac{(2-n)}2}$, where $r(x)=|x|$ is the distance function and $a,b>0$ are constants. Also for any domain $\Omg\subset \R^n$, we have $SD(\Omg)=SD(\R^n)$.

For later use, consider the annulus $D_{r_1,r_2}:=D_{r_2}(0)\setminus \bar D_{r_1}(0)\subset \R^n$ with $n\geq 4$, and we show that the optimal Sobolev constant can be approximated by $SD(D_{r_1,r_2},\psi_0)$ for some symmetric function $\psi_0(r)\in C_0^{\infty}(D_{r_1,r_2})$, small $0<r_1<1$ and big $r_2>2$. For this, take any small $\epsl>0$ and set $$\Omg=D_{r_1,r_2},\ \Omg_1=D_{r_1,2r_1},\ \Omg_2=D_{r_2-1,r_2},\ \Omg_3=D_{2r_1,r_2-1}.$$ Take a smooth cut-off function $\eta:\R_{\geq 0}\to\R_{\geq 0}$ with $$\eta|_{\Omg_3}=1,\ \supp\eta\subset \Omg,\  |\nabla \eta|\leq \frac 2{r_1}\ over\ \Omg_1,\ |\nabla\eta|\leq 2\ over\ \Omg_2.$$
Define $\psi_0=\eta(r)\cdot\fai_0(r)\in C_0^\infty(\Omg)$ for $\fai_0(r)= (1+r^2)^{\frac{2-n}{2}}$. Then 
\begin{align*}
	\int_\Omg|\nabla\psi_0|^2\dvol_{g_E}
	&\leq \int_{\Omg_3}|\nabla\fai_0|^2+\frac{C(n)}{r_1^2}\int_{r_1}^{2r_1}(1+r^2)^{2-n}r^{n-1}dr\\
	&\  +C(n)\int_{r_1}^{2r_1}(1+r^2)^{-n}r^{n+1}dr\\  &\  +C(n)\int_{r_2-1}^{r_2}((1+r^2)^{2-n}+(1+r^2)^{-n}r^2)r^{n-1}dr\\
	&\leq \int_{\Omg_3}|\nabla\fai_0|^2+C(n)r_1^{n-2}+C(n)(r_2-1)^{3-n},
\end{align*}
and $$\left(\int_\Omg\psi_0^{\frac{2n}{n-2}}\dvol_{g_E}\right)^{\frac{n-2}n}\geq \left(\int_{\Omg_3}\fai_0^{\frac{2n}{n-2}}\right)^{\frac{n-2}n}\geq \frac1{C(n)}>0.$$
Note $\fai_0$ satisfies the equation over $\R^n$ $$\Delta \fai_0+n(n-2)\fai_0^{\frac{n+2}{n-2}}=0,$$
and $$SD(\R^n)=SD(\R^n,\fai_0)=n(n-2)\left(\int_{\R^n} \fai_0^{\frac{2n}{n-2}}\dvol_{g_E}\right)^{\frac2n}.$$
Integral by parts we have 
\begin{align*} 
\int_{\Omg_3}|\nabla\fai_0|^2&
=n(n-2)\int_{\Omg_3}\fai_0^{\frac{2n}{n-2}}+\int_{\ptl D_{r_2-1}}\fai_0\frac{\ptl \fai_0}{\ptl r}-\int_{\ptl D_{2r_1}}\fai_0\frac{\ptl \fai_0}{\ptl r}\\
&\ =n(n-2)\int_{\Omg_3}\fai_0^{\frac{2n}{n-2}}+C(n)\left(\frac{(2r_1)^n}{(1+(2r_1)^2)^{n-1}}-\frac{(r_2-1)^n}{ (1+(r_2-1)^2)^{n-1}}\right).
\end{align*}
Now we first choose $r_2$ large enough such that $C(n)(r_2-1)^{3-n}\leq \frac{\epsl}2$, then choose $r_1>0$ small enough such that $C(n)r_1^{n-2}\leq \frac\epsl2$ and $$\int_{\Omg_3}|\nabla\fai_0|^2\leq n(n-2)\int_{\Omg_3}\fai_0^{\frac{2n}{n-2}},$$
so $$SD(\Omg,\psi_0)\leq (1+\epsl)SD(\R^n).$$

From these facts we can prove a gap lemma, and for this we first define the Dirichlet-Sobolev constant at infinity of a pointed manifold $(M,g,p)$ as $$SD_\infty(M,g,p):=\inf_{0<r_1<r_2}\limsup_{r\to\infty}SD(B_{r_2r}^g(p)\setminus \bar B_{r_1r}^g(p),g).$$
\begin{lem}\label{gap-lemma}
	Let $n\geq 4$ and $(M,g,p)$ be a pointed complete Ricci flat $n$-manifold with maximal volume growth, i.e. there exists $c_0>0$ such that for all $r>0$ $$\frac{\Vol_gB_r^g(p)}{\omg_nr^n}\geq c_0.$$ Then there is an $\epsl=\epsl(n)>0$ such that if  $$SD_\infty(M,g,p)\geq (1-\epsl)\cdot SD(\R^n),$$
	then $(M,g)$ is isometric to $\R^n$.
\end{lem}
\begin{proof}
	From $\Ric_g=0$ and volume comparison we know the volume ratio is non-increasing and has a upper bound by $1$. So we have $$\lim_{r\to\infty}\frac{\Vol_g\ptl B_r^g(p)}{n\omg_nr^{n-1}}=c_1$$
	for some $c_0\leq c_1\leq 1$.
	Then for any small $\epsl>0$, there exists $r_0>0$ such that for all $r\geq r_0$ $$c_1\leq \frac{\Vol_g\ptl B_r^g(p)}{n\omg_nr^{n-1}}\leq (1+\epsl)c_1.$$  
	Let $0<r_1<r_2$ be radius as discussed above and $\psi_0(r)$ be a symmetric function such that $$SD(D_{r_1,r_2},\psi_0)\leq (1+\epsl)SD(\R^n).$$ 
	Choose a large $r_{i}\geq r_0$ such that $$SD(B_{r_2r_i}^g(p)\setminus \bar B_{r_1r_i}^g(p),g)\geq (1-\epsl)SD_\infty(M,g,p),$$
	and take a test function $\fai$ over $B_{r_2r_i}^g(p)\setminus \bar B_{r_1r_i}^g(p)$ be $$\fai(x)=\psi_0\left(\frac{d_g(p,x)}{r_i}\right).$$ So from the definition of Sobolev constant we have 
	$$SD(B_{r_2r_i}^g(p)\setminus \bar B_{r_1r_i}^g(p),g)\cdot\leq (1+\epsl)^2c_1^{\frac 2n}\cdot SD(\R^n),$$
	which implies that $$c_1\geq \frac{(1-\epsl)^{ n}}{(1+\epsl)^n}.$$
	So by Lemma \ref{gap-lem-volume} it is sufficient to take $\epsl$ small enough such that $c_1\geq 1-\epsl(n)$.
\end{proof}

\subsection{Positive Yamabe constant and Sobolev constant}\label{positive-yamabe-constant and sobolev constant}
Recall that for a domain $\Omg$ in a Riemannian manifold $(M,g)$, the corresponding Yamabe functional of $\fai\in W^{1,2}_0(\Omg,g)$ is $$Y(\Omg,\fai):=\frac{\int_{\Omg}(|\nabla_{g} \fai|^2+a_0R_{g}\fai^2)\dvol_{g}}{\left(\int_{\Omg}\fai^{\frac{2n}{n-2}}\dvol_{g}\right)^{\frac{n-2}n}},$$ where $a_0=\frac{n-2}{4(n-1)}$, and the Yamabe constant is $$Y(\Omg,g)=\inf_{\substack{\fai\in W^{1,2}_0(\Omg)\\ \fai\neq 0}}Y(\Omg,\fai),$$
which is conformal invariant. And $Y(\Omg_1,g)\geq Y(\Omg_2,g)$ when $\Omg_1\subset \Omg_2$. 

Note that given a compact manifold $(M,g)$, after solving the Yamabe problem and use the conformal invariance of Yamabe functional we can assume that $R_{g}=\Const$, see e.g. in \cite{SchoenYau:1994}. Note when $R_{g}=0$, the Yamabe functional is exactly the Dirichlet-Sobolev functional and $Y(\R^n)=SD(\R^n)$. 

In general, if $Y(\Omg,g)\geq Y_0>0$, then for any $\fai\in W_0^{1,2}(\Omg,g)$ we have \begin{align*}
Y_0\cdot \left(\int_{\Omg}\fai^{\frac{2n}{n-2}}\dvol_{g}\right)^{\frac{n-2}n}&\leq \int_{\Omg}(|\nabla_{g}\fai|^2+a_0R_{g}\fai^2)\dvol_{g} \\
&\leq \int_{\Omg}|\nabla_{g}\fai|^2\dvol_{g} \\
&\ + a_0\left(\int_{\Omg}|R_{g}|^p\dvol_{g}\right)^{\frac1p}\cdot (\Vol_g \Omg)^{\frac 2n-\frac 1p}\cdot \left(\int_{\Omg}\fai^{\frac{2n}{n-2}}\dvol_{g}\right)^{\frac{n-2}n}.
\end{align*}
So if 
\begin{equation}\label{yamabe-term} a_0\left(\int_{\Omg}|R_{g}|^p\dvol_{g}\right)^{\frac1p}\cdot (\Vol_g \Omg)^{\frac 2n-\frac 1p}\leq \frac{Y_0}2, \tag{$*$}
\end{equation}
then we can get that $$SD(\Omg,g)\geq\frac{Y_0}2.$$ In the following sections we will make use of the smallness of either scalar integral or volume to ensure (\ref{yamabe-term}), and so we can get volume growth control from positive Yamabe constant. Since the Yamabe constant has invariance properties when restricted in a conformal class, it is better to get the positivity of Yamabe constant than Sobolev constant from some geometric conditions.

In the case $R_g>0$ we can easily get the positivity of Yamabe constant by solving the Yamabe problem. In general if we consider a domain with small volume in a compact manifold, then we can also get the positivity of Yamabe constant from the following lemma.
\begin{lem}\label{small-volume-domain}
	Given a $n$-dimensional compact Riemannian manifold $(M,g)$, then there exist constants $C(g)>0$ and $\epsl=\epsl(g)>0$ such that for any domain $U\subset M$ with $\Vol_gU\leq \epsl$ we have $$Y(U,g)\geq \frac1{C(g)}.$$
\end{lem}
\begin{proof}
	For any $\fai\in W_0^{1,2}(U)$ and $\fai\neq 0$, we know that
	$$\int_{U}a_0R_g\fai^2\dvol_{g}\leq C_1(g)\cdot (\Vol_gU )^{\frac 2n} \left(\int_{U}\fai^{\frac{2n}{n-2}}\dvol_{g}\right)^{\frac{n-2}n},$$ and by Sobolev inequality on $(M,g)$ 
	\begin{align*} \left(\int_{U}\fai^{\frac{2n}{n-2}}\dvol_{g}\right)^{\frac{n-2}n}&\leq C_2(g)\|\fai\|_{W^{1,2}(M,g)}^2\\
	&\leq C_2(g)\int_{U}|\nabla_{g}\fai|^2\dvol_{g}+C_2(g) (\Vol_gU )^{\frac 2n} \left(\int_{U}\fai^{\frac{2n}{n-2}}\dvol_{g}\right)^{\frac{n-2}n}.
	\end{align*}
	If we choose $\epsl=\epsl(g)$ small such that $C_2(g)\epsl^{\frac 2n}\leq \frac12$,  then
	$$\left(\int_{U}\fai^{\frac{2n}{n-2}}\dvol_{g}\right)^{\frac{n-2}n}\leq 2C_2(g)\int_{U}|\nabla_{g}\fai|^2\dvol_{g},$$
	and  $$\int_{U}(|\nabla_{g} \fai|^2+a_0R_{g}\fai^2)\dvol_{g}\geq (1-2C_1(g)C_2(g)\epsl^{\frac 2n})\int_{U}|\nabla_{g}\fai|^2\dvol_{g}.$$
	Again choose small $\epsl$ with $2C_1(g)C_2(g)\epsl^{\frac 2n}\leq \frac12$, then we can set $C(g)=4C_2(g)$ and have the desired control of Yamabe constant.
\end{proof}
When restricted in a small neighborhood of a compact manifold we can say more.
\begin{lem}\label{yamabe-rigidty}
	Given a $n$-dimensional compact Riemannian manifold $(M,g)$, then for any $\epsl>0$ there exists a $\delta=\delta(n,g)>0$ such that for any $x\in M$ $$Y(B_{\delta}^{g}(x),g)\geq (1-\epsl)Y(\R^n).$$
\end{lem}
\begin{proof}
	For any $x\in M$, there exists $\delta_x>0$ small such that there is a coordinate $\fai_x:B_{\delta_x}^g(x)\to \R^n$, under which $|g_{ij}(y)-\delta_{ij}|<\epsl$. For any $\fai\in W_0^{1,2}(B_{\delta_x}^{g}(x)),$ set $\psi=\fai\circ\fai_x^{-1}$ and $D_x=\fai_x(B_{\delta_x}^{g}(x))$, and since $$\int_{B_{\delta_x}^{g}(x)}a_0R_g\fai^2\dvol_{g}\leq C(g)\cdot (\Vol_gB_{\delta_x}^{g}(x) )^{\frac 2n} \left(\int_{B_{\delta_x}^{g}(x)}\fai^{\frac{2n}{n-2}}\dvol_{g}\right)^{\frac{n-2}n}$$
	and we can take $\delta_x$ small enough such that $ C(g)\cdot (\Vol_gB_{\delta_x}^{g}(x) )^{\frac 2n}<\epsl,$ we have
	$$\frac{\int_{B_{\delta_x}^{g}(x)}(|\nabla_{g} \fai|^2+a_0R_g\fai^2)\dvol_{g}}{\left(\int_{B_{\delta_x}^{g}(x)}\fai^{\frac{2n}{n-2}}\dvol_{g}\right)^{\frac{n-2}n}}\geq(1-\epsl)\cdot \frac{\int_{D_x}(|\nabla_{g_E} \psi|^2)\dvol_{g_E}}{\left(\int_{D_x}\psi^{\frac{2n}{n-2}}\dvol_{g_E}\right)^{\frac{n-2}n}}-\epsl.$$
	So from the fact that $Y(D_x)=Y(\R^n)$ we know the conclusion holds locally. From the compactness of $M$ it is easy to see the conclusion holds with a uniform $\delta>0$.
\end{proof}

\section{$L^p$-constant scalar curvature in conformal class of $S^n$}\label{L^p constant on S^n}
Note that when $n\geq 3$ and for conformal metrics $g_u=u^{\frac 4{n-2}}g_0$ we have the scalar curvature equation $$-\Delta_{g_0}u=-a_0R_{g_0}u+a_0R_{g_u}u^{\frac4{n-2}}\cdot u,$$
where $a_0=\frac{n-2}{4(n-1)}$.
In this section we will focus on this equation in the conformal class $\Cc(S^n,g_0)$, and use $\epsl$-regularity Lemma \ref{epsl-regularity} with blowup method to prove the rigidity result under scalar curvature rigidity condition.
\begin{proof}[Proof of Theorem \ref{L^p-constant scalar}]
It is sufficient to show that for a sequence of conformal metrics $(S^n,g_i)\in\Cc(S^n,g_0)$ with $\Vol_{g_i}S^n=\Vol_{g_0}S^n$ and $$ \lim_{i\to\infty}\int_{S^n}|R_{g_i}-n(n-1)|^p \dvol_{g_i}=0,$$ there is a subsequence converges to $(S^n,g_0)$ in the $C^{\af}$-topology, which means that up to diffeomorphisms metrics $g_i\to g_0$ in $C^{\af}(S^n)$ as tensors.

Locally over $B_1^{g_0}(x)\subset S^n$, under coordinate the scalar equation becomes $$-a^{ij}\ptl_i\ptl_ju+b^i\ptl_iu=f_1u+f_2u^{\frac4{n-2}}u,$$
where $$a^{ij}= g_0^{ij},\ b^i=\frac1{\sqrt{\det g_0}}\ptl_j(\sqrt{\det g_0}\cdot g_0^{ij}),$$
and $$C(g_0)^{-1}\leq a^{ij}\leq C(g_0),\ \|a^{ij}\|_{C^0}+ \|b^i\|_{L^{\infty}}\leq C(g_0),$$
$$f_1=-a_0R(g_0)\in C^0(M),\ f_2=a_0R(g_u)\in L^p(M,g_u).$$
So $$\|u\|_{W^{1,2}(S^n,g_0)}\leq C.$$
Write $g_k=u_k^{\frac4{n-2}}g_0$ and set \[ r(g_k,x):=\inf\{r>0: \Vol_{g_k}B^{g_0}_r(x))=\frac{\epsl_0}2\} ,\]
\[ r_k:=\inf_{x\in S^n}r(g_k,x).\]
If there exists some $ r_0>0 $ such that \[ r_k\geq r_0 ,\] then from $\epsl$-regularity Lemma \ref{epsl-regularity} we get the $ W^{2,p} $-weak convergence and thus $C^{\af}$-convergent subsequence of $u_k$. So let's assume that for some sequence $$x_k\to x_0\in S^n,\ r_k=r(g_k,x_k)\to 0. $$
We will show that the blow up at the concentration point $x_0$ gives diffeomorphisms of $S^n$, which allows us to get convergence up to such diffeomorphisms. First through a rotation $\sigma_k$ of $S^n$ we can assume that $x_k=x_0$ for all $k$. Then let $y_0$ be the antipodal point of $x_0$ and $$\pi:S^n\setminus\{y_0\}\to \R^n$$ the stereographic projection map with $\pi(x_0)=0$. We still denote the metric $(\pi^{-1})^*g_0$ by $g_0$ which is conformal to the Euclidean metric $g_E$ on $\R^n$.
Then there exists a constant $b_0>0$ such that at point $0$ $$g_0(0)=b_0g_E.$$
Let $y$ be the coordinate of $\R^n$ and we shall consider the scalar dilations over $\R^n$. For this we set
\[ g_{0,k}(x) :=g_{0,ij}(r_kx)dx^i\otimes dx^j,\ \tilde u_k(x):=r_k^{\frac{n-2}2}u_k(r_k x),\]
then under linear rescale $ y=r_kx $, in these different coordinates we have the relation \[ (D(a,r),g_{0,k}(x)) = (D(ar_k,rr_k),r_k^{-2}g_0(y)),\]
%\[ \tilde R_k(x) =r_kR_k(x_k+r_kx),\ dV_{g_{0,k}}(x)=dV_g(x_k+r_kx),\]
where $D$ means the disk in $\R^n$ for corresponding coordinate, and in coordinate $x$ the scalar equation becomes
\[ -\Delta_{g_{0,k}}\tilde u_k(x)=-a_0R_{g_{0,k}}\tilde u_k(x)+a_0R_{g_k}(r_kx)\tilde u_k^{\frac{n+2}{n-2}}(x) .\]
Now for any fixed $r>0$ and any $ a\in D(0,r) $ , for all large $k$ we have
$$ \int_{D(a,1)}|R_{g_k}(r_kx)|^p\tilde u_k^{\frac{2n}{n-2}}(x)\dvol_{g_{0,k}}(x)
=\int_{D(ar_k,r_k)}|R_{g_k}(y)|^pu_k^{\frac{2n}{n-2}}(y)\dvol_{g_0}(y)\leq C, 
$$
\[ \int_{D(a,1)}|\tilde u_k|^{\frac{2n}{n-2}}\dvol_{g_{0,k}}(x)=\int_{D(ar_k,r_k)}|u_k|^{\frac{2n}{n-2}}\dvol_{g_0}(y)\leq \epsl_0, \] where the last inequality is from our definition $ r_k\leq r(g_k,ar_k) $ and note the last integral is equal to $\frac{\epsl_0}2$ when $ a=0 $. Since we can find finite cover of $D(0,r)$ by $D(a,1)$, and $g_{0,k}\to b_0g_{E}$ smoothly as $ k\to \infty $, for any $r>0$ from $\epsl$-regularity we have
 \[ \|\tilde u_k\|_{W^{2,p}(D(0,r),g_{E})} \leq C(r).\]
So there exists a function $v\geq 0$ over $ \R^n $ such that up to subsequence, $$ \tilde u_k\wto v \ in\  W^{2,p}_{loc}(\R^n).$$ Let $k\to \infty$ then $v$ satisfies the equation \[ -\Delta_{b_0g_E} v= a_0n(n-1)v^{\frac{n+2}{n-2}} .\]
Note that pullback this equation by $\pi$ gives the scalar equation with constant scalar $n(n-1)$ on $S^n$, and also \[\frac{\epsl_0}2= \int_{D(0,1)}v^{\frac{2n}{n-2}}dvol_{g_E}\leq\int_{\R^n}v^{\frac{2n}{n-2}}\dvol_{g_E}\leq C, \]
which implies $v$ is non-zero and thus gives the standard metric $g_0$. So through $\pi$ we
get diffeomorphisms $f_k$ on $S^n$ that $$f_k(x)=\pi^{-1}\left(\frac{\pi(x)}{r_k}\right),$$ and $f_k^*g_k$ converges to $g_0$ weakly in $W^{2,p}_{loc}(S^n-\{y_0\})$. In fact this weak convergence is global, since otherwise $y_0$ will be a concentration point, then for any small $r>0$, \[ \liminf_{k\to \infty}\Vol_{g_k}B^{g_0}_r(y_0)\geq \frac{\epsl_0}2>0, \] so 
\begin{align*} 
	\liminf_{k\to \infty}\Vol_{g_k}(S^n)&\geq \lim_{r\to 0}\liminf_{k\to \infty}\Vol_{g_k}(S^n-B^{g_0}_r(y_0))+ \frac{\epsl_0}2\\
	&=\Vol_{g_0} S^n+\frac{\epsl_0}2,
\end{align*} 
which contradicts the volume rigidity condition.
\end{proof}
\section{$L^p$-bounded sectional curvature with positive Yamabe constant}\label{section-sectional yamabe}
In this section we use blowup argument of harmonic radius to get a compactness result under $L^p$-bounded sectional curvature and additional positive Yamabe constant condition. Then we give a proof of the Mumford's lemma.
\begin{proof}[Proof of Theorem \ref{positive-yamabe}]
It is sufficient to show for any $n<q<p^*=\frac{np}{n-p}$ the $W^{1,q}$-harmonic radius of $(M_i,g_i)$ has a uniform lower bound, then the theorem follows from the compactness theorem Lemma \ref{compactness}. We argue by contradiction and say there exists a subsequence $(M_i,g_i)$ with $W^{1,q}$-harmonic radius $\epsl_i=\epsl(x_i)\to 0$. Consider the pointed compact manifolds $(M_i,\tilde g_i=\epsl_i^{-2}g_i,x_i)$ with a uniform harmonic radius $\tilde \epsl(x_i)=1$, then by Lemma \ref{compactness}, there exists a complete $C^{\af}$-Riemannian manifold $(N,h)$ such that up to subsequence in the $C^{\af}$-topology $$(M_i,\tilde g_i,x_i)\to (N,h,x_\infty).$$
Then for any $r>0$ there exists a domain $B_r^h(x_\infty)\subset\Omg_r$ and diffeomorphisms $$f_i:\Omg_r\to M_i$$ with $B_r^{\tilde g_i}(x_i)\subset f_i(\Omg_r)$, $f_i^*\tilde g_i$ are uniformly bounded in $W^{1,q}(\Omg_r,h)$ and $f_i^*\tilde g_i\to h$ in $C^{\af}(\Omg_r, h)$. From the elliptic equation of metric involving Ricci curvature in the harmonic coordinate, and  $$\int_{M_i}|K_{\tilde g_i}|^p\dvol_{\tilde g_i}=\epsl_i^{2p-n}\int_{M_i}|K_{g_i}|^p\dvol_{g_i}\leq \Lbd\cdot \epsl_i^{2p-n}\to 0,$$
by the $L^p$-estimate Lemma \ref{L^p-estimate} we in fact have $W^{2,p}$-weak convergence of metrics, then over $(N,h)$ the metric satisfies a weak elliptic equation with $K_h=0$, which together with regularity of elliptic equations implies that $h$ is a smooth metric. 

Note that $N$ is non-compact, since otherwise we may take $r>2\diam_hN$ and then all embeddings $f_i:N\to M_i$ are both open and closed, so $f_i(N)=M_i$ and thus $$\Vol_hN=\lim_{i\to\infty}\Vol_{\tilde g_i}M_i\geq \lim_{i\to\infty}\epsl_i^{-n}V_0=\infty,$$
a contradiction with the assumption that $N$ is compact. So for any $r>0$, $\ptl B_r^h(x_\infty)\neq \emptyset$ and then $\ptl B_r^{\tilde g_i}(x_i)\neq \emptyset$ for all large $i$.

We claim that $(N,h)=\R^n$, which gives the desired contradiction since by assumption the maximal harmonic radius at $x_\infty$ is $1$ but $\R^n$ has global harmonic coordinates. It is sufficient to show that $(N,h)$ has maximal volume growth, which together with $K_h=0$ implies that $N$ is simply connected and thus the Euclidean space. For this, take any $r>0$ and note that
$$\frac{\Vol_h B^h_r(x_\infty)}{\omg_nr^n}=\lim_{i\to \infty}\frac{\Vol_{\tilde g_i} B_r^{\tilde g_i}(x_i)}{\omg_nr^n}\leq 1.$$
Since $\tilde g_i$ is conformal to $g_i$, for all $i$ large enough with $r\epsl_i\leq s_0$ we have $$Y(B_r^{\tilde g_i}(x_i),\tilde g_i)=Y(B_{r\epsl_i}^{ g_i}(x_i),g_i)\geq Y_0.$$
Also (\ref{yamabe-term}) is satisfied for large $i$ since $$\left(\int_{B_r^{\tilde g_i}(x_i)}|R_{\tilde g_i}|^p\dvol_{\tilde g_i}\right)^{\frac1p}\cdot (\Vol_{\tilde g_i} B_r^{\tilde g_i}(x_i))^{\frac 2n-\frac 1p}\leq \epsl_i^{2-\frac np}\cdot \Lbd^{\frac 1p}\cdot (2\omg_nr^n)^{\frac 2n-\frac 1p} ,$$
so for all $i$ large enough $$SD(B_r^{\tilde g_i}(x_i))\geq \frac{Y_0}2,$$
which together with Lemma \ref{volume-ratio} implies that $B_r^{\tilde g_i}(x_i)$ has maximal volume growth, thus $$\frac{\Vol_h B^h_r(x_\infty)}{\omg_nr^n}\geq C(Y_0)>0.$$
\end{proof}
As a corollary we can prove the generalization of Mumford's lemma.
\begin{proof}[Proof of Theorem \ref{Mumford's lemma}]
	Assume that $\pi_1(M)$'s minimal length $\geq l_0>0$. Let $\tilde M$ be its universal covering with induced metric from $M$, so $\tilde M$ is simply connected locally conformal flat manifolds, and then there exists a conformal immersion $\Fai:\tilde M\to S^n$, which implies that $Y(\tilde M)=Y(S^n,g_0)>0$, see the details and related results in \cite{SchoenYau:1994}.
	
	We claim that $\pi: B_{\frac{l_0}2}(\tilde p)\to B_{\frac{l_0}2}(p)$ is an isometry. To see this, it is only need to show $\pi|_{B_{\frac{l_0}2}(\tilde p)}$ is injective, otherwise say there exist $\tilde x_1, \tilde x_2\in B_{\frac{l_0}2}(\tilde p)$ with $\pi(\tilde x_1)=\pi(\tilde x_2)=x$. On one hand, $$d(\tilde x_1,\tilde x_2)\leq d(\tilde x_1,\tilde p)+d(\tilde p,\tilde x_2)<l_0.$$
	On the other hand, choose a minimal geodesic $\tilde \gm$ connect $\tilde x_1$ and $\tilde x_2$. Then set $\gm=\pi(\tilde \gm)$ which is a non-trivial loop at $x$, so the length of $\gm$ is greater than $l_0$. The local isometry property of $\pi$ implies that the length of $\tilde\gm$ and thus $d(\tilde x_1,\tilde x_2)$ is greater than $l_0$, a contradiction. 
	
	So we have uniform positive Yamabe constant over the geodesic ball with uniform radius $\frac{l_0}2$. Thus this theorem is a corollary of the above theorem.
\end{proof}

\section{$L^p$-bounded sectional curvature in conformal class}\label{section-L^p-bounded-sectional}
With those techniques discussed before, we can now prove our main results. To better understand the ideas, we first prove the sectional curvature case in this section.
\begin{proof}[Proof of Theorem \ref{L^p-section}]
By Solving the Yamabe equation we can assume that $R_{g_0}=R_0=\Const$. It is sufficient to show for any $n<q<p^*$ the $W^{1,q}$-harmonic radius has a uniformly lower bound. By contradiction argument same as before in Section \ref{section-sectional yamabe} for the proof of Theorem \ref{positive-yamabe}, there exist $g_i\in\Cc(M,g_0)$ with the harmonic radius $\epsl_i=\epsl(x_i)\to 0$ and in $C^{\af}$-topology $$(M,\tilde g_i=\epsl_i^{-2}g_i,x_i)\to (N,h,x_\infty),$$
where $(N,h)$ is a complete non-compact Riemannian manifold with $K_h=0$. And we also need to show that $(N,h)=\R^n$ which will give a contradiction.

Case 1): for any $r>0$ we have $$\Vol_{g_0}B_r^{\tilde g_i}(x_i)\to 0$$ as $i\to\infty.$
Choose small $\epsl_1=\epsl_1(g_0)$ appeared in Lemma \ref{small-volume-domain}, then for all large enough $i$, we have $$\Vol_{g_0}B_r^{\tilde g_i}(x_i)\leq \epsl_1.$$
By Lemma \ref{small-volume-domain} there exists a positive constant $C(g_0)>0$ such that $$Y(B_r^{\tilde g_i}(x_i),\tilde g_i)=Y(B_r^{\tilde g_i}(x_i),g_0)\geq \frac1{C(g_0)}>0.$$
The remaining argument is the same as in Section \ref{section-sectional yamabe}, where we get a uniform Sobolev constant from this positive Yamabe constant and thus the maximal volume growth.

Case 2): there exists $\ubar r_0>0$ and $\ubar V_0>0$ and a subsequence with $$\Vol_{g_0}B_{\ubar r_0}^{\tilde g_i}(x_i)\geq \ubar V_0>0.$$ We will use the scalar equation to get some convergence of functions.
For this we shall pullback all geometric informations on $(M,\tilde g_i)$ to the limit manifold $(N,h)$ and consider the equations under coordinate of $N$. Now taking an increasing sequence $s_i\to \infty$ and set $B_i^h=B_{s_i}^h(x_\infty)\subset N$. Then for any $i$ we can find embeddings $f_{k(i)}:B_i^h\subset \Omg_i\to (M,\tilde g_{k(i)})$ such that $$d_h(x_\infty,f_{k(i)}^{-1}(x_{k(i)}))\leq 2^{-i},\ \|f_{k(i)}^* \tilde g_{k(i)}-h\|_{C^{\af}(B_i^h)}\leq 2^{-i}.$$  
For simplicity we still denote $f_{k(i)}$ by $f_i$. In the following we set $g_0=v_i^{\frac 4{n-2}}\tilde g_i$, $\hat g_i=f_i^*\tilde g_i$ and $g_{0,i}=f_i^*g_0= w_i^{\frac 4{n-2}}\hat g_i,$ with $w_i=v_i\circ f_i.$ Note that
over $M$, $v_i$ satisfies the equation $$-\Delta_{\tilde g_i}v_i=-a_0R_{\tilde g_i}v_i+a_0R_0v_i^{\frac4{n-2}}\cdot v_i.$$ Then for any fixed $r>0$, pullback by $f_i$ for all large $i$, we have equation over $B_r^h:=B_r^h(x_\infty)$
$$-\Delta_{\hat g_i} w_i=-a_0R_{\hat g_i} w_i+a_0R_{0} w_i^{\frac4{n-2}}\cdot  w_i.$$
Choose local coordinate $\{x^j\}$ on $B_1^h(x)\subset B_r^h$, this equation becomes $$-a^{jk}_i\ptl_j\ptl_kw_i+b^j_i\ptl_jw_i=-a_0R_{\hat g_i} w_i+a_0R_{0} w_i^{\frac4{n-2}}\cdot  w_i,$$
where $$a_i^{jk}=\tilde  g_i^{jk},\ b_i^j=\frac1{\sqrt{\det \tilde g_i}}\ptl_k(\sqrt{\det \tilde g_i}\cdot\tilde g_i^{jk})$$
with $$\|a_i^{jk}\|_{C^0}+ \|b_i^j\|_{L^q}\leq C,$$
and $$\int_{N}|R_{\hat g_i}|^p\dvol_{h}\leq \Lbd\epsl_i^{2p-n},\ \int_{N}w_i^{\frac{2n}{n-2}}\dvol_h\leq C(g_0). $$
So there exists $0\leq w\in W^{1,2}(N)$ and up to subsequence $$w_i\wto w\ in\ W^{1,2}_{loc}(N).$$
Now we consider separately the two cases whether $R_0\leq 0$ or $R_0>0$.

Case 2i): $R_0\leq 0$. In this case we have a differential inequality $$-\Delta_{\hat g_i} w_i\leq -a_0R_{\hat g_i} w_i.$$
From naive point of view, $R_{\hat g_i}$ is almost $0$ and such inequality says that the limit $w$ is subharmonic on $(N,h)$, which together with $\Ric_h=0$ implies that $w$ is a constant. Since $w_i^{\frac4{n-2}}\hat g_i$ are all compact metrics $g_0$, the limit $w^{\frac4{n-2}}h$ is again compact which contradicts with fact that $(N,h)$ is non-compact. We transfer this view into strict language in the following.

For any $0<r<R$, take a cut-off function $0\leq\eta_i\leq 1$ with $\supp\eta_i\subset B_R^{\hat g_i}$ and $\eta_i|_{B_r^{\hat g_i}}=1$, $|\nabla_{\hat g_i}\eta|\leq \frac2{R-r}$, multiplying the above differential inequality with $\fai_i=\eta_i^2 w_i^{\af}$ and integral to get for all large $i$
\begin{align*}
\int_{B_r^h}|\nabla_{h} w_i^{\frac{\af+1}2}|^2\dvol_{h}
&\leq C\int_{B_R^{\hat g_i}}(|\nabla_{\hat g_i}\eta|^2 w_i^{\af+1}+R_{\hat g_i} w_i^{\af+1})\dvol_{\hat g_i}\\
&\leq C\left(\int_{B_R^{\hat g_i}} w_i^{(\af+1)\frac{p}{p-1}}\dvol_{\hat g_i}\right)^{1-\frac1p}\left(\frac{(\Vol_{h}B_R^{h})^{\frac1p}}{(R-r)^2}+\left(\int_{B_R^{\hat g_i}}|R_{\hat g_i}|^p\dvol_{\hat g_i}\right)^{\frac1p}\right).
\end{align*}
Also note that $$\int_{B_r^h}| w_i^{\frac{\af+1}2}|^2\dvol_h\leq C\left(\int_{B_{2r}^{\hat g_i}} w_i^{(\af+1)\frac{p}{p-1}}\dvol_{\hat g_i}\right)^{1-\frac1p}\cdot(\Vol_hB^h_r)^{\frac1p}.$$
If we choose $(\af+1)\frac p{p-1}=\frac{2n}{n-2}$, then
 $$\| w_i^{\frac{\af+1}2}\|_{W^{1,2}(B_r^h,h)}\leq C(h,R,\Lbd,p,n,g_0).$$
So we can take subsequence such that $$ w_i^{\frac{\af+1}2}\rightharpoondown  w^{\frac{\af+1}2}\ in\ W^{1,2}(B_r^h,h),$$
and note that $\frac{\af+1}2>1$ so we can take subsequence with $$ w_i\to  w\ in\ L^{\frac{2n}{n-2}},\  w_i\to  w\ a.e.$$
Taking $i\to \infty$ in above inequality, and note $K_h=0$ implies volume comparison that $\Vol_hB_R^h\leq  \omg_nR^n$, we have 
$$\int_{B_r^h}|\nabla_{h} w^{\frac{\af+1}2}|^2\dvol_{h}\leq C(g_0,p,n)\frac{R^{\frac np}}{(R-r)^2},$$
then let $R\to \infty$ we have $|\nabla_{h} w^{\frac{\af+1}2}|=0$ over $B_r^h$. By the arbitrary of $r$ and diagonal method we have $$w_i\to w=\Const\ in\
L^{\frac{2n}{n-2}}_{loc}(N,h).$$
From the assumed positive volume condition on a uniform ball, we have $$\int_{B_{\ubar r_0}^{\hat g_i}(x_\infty)} w_i^{\frac{2n}{n-2}}\dvol_{\hat g_i}\geq \ubar V_0>0.$$
Let $i\to \infty$ we know $w>0$ is a positive constant.
Now from $\epsl$-regularity Lemma \ref{epsl-regularity} we know that up to subsequence $$w_i\wto w\ in\ W^{2,p}_{loc}(N,h),$$
and thus $$w_i\to w\ in\ C^{\af}_{loc}(N,h).$$
Let $c_0=w^{\frac4{n-2}}$, then $f_i^*g_0\to c_0h$ in $C^{\af}_{loc}(N)$. Since for all $x,y\in N$ $$d_{f_i^*g_0}(x,y)\leq \diam_{g_0}(M)<\infty,$$
and $C^{\af}$-convergence ensures the convergence of distance, we have $$\diam_{c_0h}(N)<\infty,$$ which is a contradiction with the fact that $(N,h)$ is non-compact.

Case 2ii): $R_0>0$. In this case we have for any $r>0$ and $x\in M$, $$Y(B_r^{\tilde g_i}(x),\tilde g_i)=Y(B_r^{\tilde g_i}(x),g_0)\geq Y(M,g_0)=:Y_0>0,$$
and also by the same argument in Section \ref{section-sectional yamabe} we can show that $(N,h)$ has maximal volume growth and thus being $\R^n$.
\end{proof}

\section{$L^p$-bounded Ricci curvature in conformal class}
With the similar argument for sectional curvature case and additional analysis, we can now prove our main result in the Ricci curvature case.
\begin{proof}[Proof of Theorem \ref{L^p-bounded ricci}]
We apply the same contradiction argument as in Section \ref{section-L^p-bounded-sectional} for the proof of Theorem \ref{L^p-section}, except that we have only a Ricci flat limit space. So we will mainly show that the limit space is in fact flat. Using the same argument and notation as before, say in the $C^{\af}$-topology $(M,\tilde g_i=\epsl_i^{-2}g_i,x_i)\to (N,h,x_\infty)$ with $\Ric_h=0$ and $f_i:B_{s_i}^h(x_\infty)\subset N\to (M,\tilde g_i)$ the corresponding diffeomorphisms with $\|f_i^*\tilde g_i-h\|_{C^{\af}(B_{s_i}^h)}\leq 2^{-i}$.
The Case 2i) before makes no difference here, so we only consider the other two cases.

Case 1): when $\Vol_{g_0}B_r^{\tilde g_i}(x_i)\to 0$ for any $r>0$, we also know that $(N,h)$ has maximal volume growth as before, so if we can show $K_h=0$ here then it is done. In fact we can use the conformal property of Weyl tensor to get this result.
Note that when $n\geq 4$ and for any $r>0$ we have
\begin{align*}
\int_{B_r^{\tilde g_i}(x_i)}|W_{\tilde g_i}|^{\frac n2}\dvol_{\tilde g_i}&=\int_{B_r^{\tilde g_i}(x_i)}|W_{g_0}|^{\frac n2}\dvol_{g_0}\\
&\leq \left(\int_M |W_{g_0}|^p\dvol_{g_0}\right)^{\frac1p}\cdot (\Vol_{g_0}B_r^{\tilde g_i}(x_i))^{1-\frac {n}{2p}},
\end{align*}
and also
\begin{align*} 
\int_{B_r^h(x_\infty)}|W_{h}|^{\frac n2}\dvol_{h}
\leq C(r,h)\liminf_{i\to\infty}\int_{B_r^{\hat g_i}(x_i)}|W_{\hat g_i}|^{\frac n2}\dvol_{\hat g_i}.
\end{align*}
Let $i\to \infty$ we know that $$W_h=0.$$
This together with $\Ric_h=0$ implies that $K_h=0$.

Case 2ii): $R_0>0$ and there exists $\ubar r_0>0$, $\ubar V_0>0$ and a subsequence with $\Vol_{g_0}B_{\ubar r_0}^{\tilde g_i}(x_i)\geq \ubar V_0>0$. Also we know $(N,h)$ has maximal volume growth.

Recall the scalar equation discussed in Case 2) of Section \ref{section-L^p-bounded-sectional}, denote the concentration set by $$\Ss=\{x\in N: \lim_{r\to 0}\liminf_{i\to \infty}\int_{B_r^{h}(x)}w_i^{\frac{2n}{n-2}}\dvol_{h}> \frac{\epsl_0}2\},$$
which is a finite set from the fact that $g_{0,i}=f_i^*g_0$ has uniform volume bound.
By Lemma \ref{epsl-regularity} we know up to subsequence $$w_i\wto w\ in\ W^{2,p}_{loc}(N\setminus\Ss).$$
Take a $r_0>0$ large such that the concentration set $\Ss\subset B_{r_0}^h$, then up to subsequence $w_i\wto w$ weakly in $W_{loc}^{2,p}(N_{r_0},h)$ where $N_{r_0}:=N\setminus\bar B_{r_0}^h$. And over $N_{r_0}$, $w\geq 0$ satisfies the weak equation $$-\Delta_hw=a_0R_0w^{\frac4{n-2}}\cdot w,$$ then $w$ is smooth and from maximum principle we know $w>0$. Denote $g_\infty:=w^{\frac 4{n-2}}h$, then $g_{0,i}\to g_\infty$ in $C^{\af}_{loc}(N_{r_0})$ and weakly $W^{2,p}_{loc}(N_{r_0})$. Similarly like before, we may choose subsequence and still denote by $f_i$ such that $$\|f_i^*g_0-g_\infty\|_{C^0(N_{r_0,s_i})}\leq 2^{-i},$$ with $N_{r_0,s_i}=B_{s_i}^h\setminus \bar B_{r_0}^h $. So $g_\infty$ is a smooth metric over $N_{r_0}$ with constant scalar $R_0$, and $$\int_{N_{r_0}}|K_{g_\infty}|^p\dvol_{g_\infty}\leq C(g_0),\ \Vol_{g_\infty}N_{r_0}\leq C(g_0).$$
Intuitively, we can recognize $g_\infty$ as the compact metric $g_0$ and then the infinity of $N$ is in some sense a small neighborhood of some point in $(M,g_0)$, which is almost the Euclidean domain and thus gives our desired rigidity condition.

For this now consider a blow down sequence of pointed manifolds $(N,h_i:=s_i^{-2}h,x_\infty)$ with $s_i\to\infty$ the increasing sequence chosen before. Since $\Ric_{h_i}=0$, there exists a metric space $(N_\infty,d_\infty)$ such that up to subsequence in the Gromov-Hausdorff topology $$(N,h_i,x_\infty)\to (N_\infty,d_\infty,o_\infty).$$ 
We claim that for any $r_1<r_2$, over $N_{r_1,r_2}^{h_i}:=\bar B_{r_2}^{h_i}(x_\infty)\setminus B_{r_1}^{h_i}(x_\infty)$ we have the injectivity radius control $$\inf_{y\in N_{r_1,r_2}^{h_i}}\frac{\inj_{h_i}(y)}{d_{h_i}(y,\ptl N_{r_1,r_2}^{h_i})}\geq i_0(r_1,r_2)>0.$$
Otherwise there exists some $r_1<r_2$ and $$\frac{\inj_{h_i}(y_i)}{d_{h_i}(y_i,\ptl N_{r_1,r_2}^{h_i})}=\inf_{y\in N_{r_1,r_2}^{h_i}}\frac{\inj_{h_i}(y)}{d_{h_i}(y,\ptl N_{r_1,r_2}^{h_i})}=:\tau_i\to 0.$$
Set $c_i:=\inj_{h_i}(y_i)$ and $\Omg_i:=N_{r_1,r_2}^{h_i}\subset N_{r_0}$ for all large $i$. Consider the sequence $(\Omg_i,\tilde h_i:=c_i^{-2}h_i,y_i)$, then $\inj_{\tilde h_i}(y_i)=1$ and $d_{\tilde h_i}(y_i,\ptl\Omg_i)\to\infty$. For any $r>0$ and $d_{\tilde h_i}(y,y_i)\leq r$, we have by definition of $y_i$ that
$$\inj_{\tilde h_i}(y)=\frac{\inj_{h_i}(y)}{c_i}\geq \frac{d_{h_i}(y,\ptl \Omg_i)}{d_{h_i}(y_i,\ptl \Omg_i)}=\frac{ d_{\tilde h_i}(y,\ptl \Omg_i)}{ d_{\tilde h_i}(y_i,\ptl \Omg_i)}\geq 1-\frac{r}{ d_{\tilde h_i}(y_i,\ptl \Omg_i)},$$ then for all $i\geq i(r)$ big enough the RHS has a uniform lower bound. Note $\Ric_{\tilde h_i}=0$, so by Lemma \ref{injectivity} we have $C^{1,\af}$-convergence subsequence say $$(\Omg_i,\tilde h_i,y_i)\to (\Omg_\infty,\tilde h_\infty,y_\infty).$$
Then $(\Omg_\infty,\tilde h_\infty)$ is a complete non-compact Ricci flat manifold with maximal volume growth. Also since $\tilde h_i$ is conformal to $g_\infty$, for any $r>0$ we have 
\begin{align*}\int_{B_r^{\tilde h_i}(y_i)}|W_{\tilde h_i}|^{\frac n2}\dvol_{\tilde h_i}
&=\int_{B_r^{\tilde h_i}(y_i)}|W_{g_\infty}|^{\frac n2}\dvol_{g_\infty}\\
&\leq \left(\int_{B_{rc_is_i}^h(y_i)}|W_{g_\infty}|^p\dvol_{g_\infty}\right)^{\frac n{2p}}\cdot (\Vol_{g_\infty}B_{rc_is_i}^h(y_i))^{1-\frac n{2p}}.
\end{align*}
Note that $g_{0,i}\wto g_\infty$ in $W^{2,p}_{loc}(N_{r_0})$ ensures that the bounded $L^p$-norm of Weyl tensor $$\|W_{g_\infty}\|_{L^p(N_{r_0},g_\infty)}\leq C(g_0),$$ and for all $i$ large enough $$B_{rc_is_i}^h(y_i)=B_{r\tau_i d_h(y_i,\ptl\Omg_i)}^h(y_i)\subset B_{\frac12 d_h(y_i,\ptl\Omg_i)}^h(y_i)\subset\Omg_i,$$
and as $i\to\infty$ $$\Vol_{g_\infty}\Omg_i\leq\Vol_{g_\infty}(\bar B_{r_2s_i}^{h}(x_\infty)\setminus B_{r_1s_i}^{h}(x_\infty)) \to 0.$$ So let $i\to\infty$ we get that $W_{\tilde h_\infty}=0$. This implies that $K_{\tilde h_\infty}=0$ thus $\Omg_\infty=\R^n$, a contradiction with the injectivity radius.

By Lemma \ref{local-compactness} we know that $N_\infty\setminus\{o_\infty\}$ is a manifold with a $C^{1,\af}$-metric $h_\infty$ which is compatible with the distance $d_\infty$ and $h_i\to h_\infty$ in $C^{1,\af}_{loc}(N_\infty\setminus\{o_\infty\})$. As done before let $F_i$ be the corresponding diffeomorhisms and set $g_\infty=w^{\frac 4{n-2}}h=\tilde w_i^{\frac 4{n-2}}h_i$ with $\tilde w_i=s_i^{\frac{n-2}2}w$, $\hat h_i=F_i^*h_i$ and $g_\infty=\hat w_i^{\frac4{n-2}}\hat h_i$ with $\hat w_i=\tilde w_i\circ F_i$, then we have the equation over $N_{r_0}$ $$-\Delta_{h_i}\tilde w_i=a_0R_0\tilde w_i^{\frac 4{n-2}}\tilde w_i,$$
which pullback by $F_i$ becomes equation over $F_i^{-1}N_{r_0}\subset N_\infty\setminus\{o_\infty\}$ $$-\Delta_{\hat h_i}\hat w_i=a_0R_0\hat w_i^{\frac 4{n-2}}\hat w_i.$$ Note that for any $r_1<r_2$, the $C^1$-norm of $h_i$ are bounded and $$\int_{N_{r_1,r_2}^{h_i}}\tilde w_i^{\frac{2n}{n-2}}\dvol_{h_i}=\int_{N_{r_1s_i,r_2s_i}^h}w^{\frac{2n}{n-2}}\dvol_h\to 0,$$
so for all $i$ large the $\epsl$-regularity Lemma \ref{epsl-regularity} can be applied, and pullback to $N_\infty$ we know that up to subsequence there exists $\hat w_\infty\geq 0$ such that $$\hat w_i\wto \hat w_\infty\ in\ W^{2,p}_{loc}(N_\infty\setminus\{o_\infty\}),$$ and $$\int_{N_\infty\setminus\{o_\infty\}}\hat w_\infty^{\frac{2n}{n-2}}\dvol_{h_\infty}=0,$$
which implies that $\hat w_\infty=0$. Thus for any fixed $r_1<r_2$ and small $\delta_0>0$ determined below, for all large enough $i$ we have $$\diam_{g_\infty}N_{r_1,r_2}^{h_i}\leq \frac12 \delta_0.$$

Now we use the diffeomorphisms $f_i$ to identify $g_0$ with $g_\infty$. For any $0<r_1< r_2$, up to subsequence and still denoted by $f_i$, then for all large $i$ we know $$\|f_i^*g_0-g_\infty\|_{C^0(N_{r_1,r_2}^{h_i})}\leq 2^{-i}.$$
For all large $i$, choose a base point $y_i\in N^{h_{i}}_{r_1,r_2}$, then $z_i:=f_i(y_i)\in M$ has a concentration point say $z_\infty\in M$. Then for any large $i$ and any $y\in N^{h_{i}}_{r_1,r_2}$ we have $$d_{g_0}(f_i(y),z_\infty)\leq d_{g_0}(f_i(y),f_i(y_i))+d_{g_0}(f_i(y_i),z_\infty)\leq \delta_0,$$ which means that the domain $$f_i(N_{r_1,r_2}^{h_i})\subset B_{\delta_0}^{g_0}(x_\infty).$$
Since $(M,g_0)$ is compact, for any $\epsl_0=\epsl_0(n)>0$ small to be determined below, by Lemma \ref{yamabe-rigidty} we can choose $\delta_0=\delta_0(g_0)$ small enough such that $$Y(B_{\delta_0}^{g_0}(z_\infty),g_0)\geq (1-\epsl_0)Y(\R^n).$$
So for all large $i$ $$Y(f_i(N^{h_{i}}_{r_1,r_2}),g_0)\geq (1-\epsl_0)Y(\R^n),$$
pullback this by $f_i$ then we have for all large $i$ $$Y(N^{h}_{r_1s_{i},r_2s_{i}},h)=Y(N^{h_{i}}_{r_1,r_2},g_\infty)\geq (1-2\epsl_0)Y(\R^n).$$
Since $\Ric_{h}=0$ this implies that for any $0<r_1<r_2$ and a subsequence $s_{i}\to\infty$ we have $$SD(N^{h}_{r_1s_{i},r_2s_{i}},h)\geq (1-\epsl_0)Y(\R^n).$$
By Lemma \ref{gap-lemma} we can choose $\epsl_0$ small and thus $(N,h)=\R^n$, a contradiction.
\end{proof}

\bibliographystyle{aomalpha}
%\printbibliography
\bibliography{reference.bib}

\end{document}